\documentclass[11pt]{amsart}
\usepackage{amsmath,amsthm}
\usepackage {latexsym,color}
\usepackage{amssymb}

\newcommand{\lesss}{\lesssim_S}

\newcommand{\const}{\mbox{\rm const}}

\newcommand{\eps}{{\varepsilon}}
\newcommand{\R}{{\mathbb R}}

\newcommand{\Z}{{\mathbb Z}}
\renewcommand{\S}{{\mathbb S}}

\newcommand{\la}{\langle}
\newcommand{\ra}{\rangle}

\renewcommand{\ln}{\log}
\newcommand\FH{{\mathcal F}_H}
\newcommand\FtH{{\mathcal F}_{\tilde H}}

\newcommand{\les}{\lesssim}
\newcommand{\ges}{\gtrsim}
\newcommand{\ga}{\gamma}
\newcommand{\tu}{\epsilon}
\newcommand{\tw}{\gamma}

\newcommand{\bu}{\bar u} 
\newcommand{\bw}{\bar w} 
\newcommand{\dHe}{{\dot{H}_e^1}}
\newcommand{\He}{{H_e^1}}

\newtheorem{theorem}{Theorem}
\newtheorem{lemma}[theorem]{Lemma}
\newtheorem{defi}[theorem]{Definition}

\newtheorem{prop}[theorem]{Proposition}
\newtheorem{proposition}[theorem]{Proposition}
\newtheorem*{conj}{Conjecture}

\newtheorem*{open}{Open Problem}
\theoremstyle{remark}
\newtheorem{remark}[theorem]{Remark}

\numberwithin{theorem}{section}

\newcommand{\calE}{{\mathcal E}}

\renewcommand{\Re}{\,{\rm Re}\,}
\renewcommand{\hat}{\widehat}
\renewcommand{\epsilon}{\eps}
\numberwithin{equation}{section}
\numberwithin{theorem}{section}

\begin{document}

\title{A codimension two stable manifold of near soliton
equivariant wave maps}

\author{I. Bejenaru}
\address{ Department of Mathematics,
  University of Chicago, 5734 S. University Ave, Chicago, IL 60637}
\email{bejenaru@math.uchicago.edu}

\author{J.\ Krieger}
\address{EPFL, Batiment des MathŽmatiques,
Station 8,
CH-1015 Lausanne}
\email{joachim.krieger@epfl.ch}

\author{D.\ Tataru}
\address{Department of Mathematics, The University of California at Berkeley, Evans Hall, Berkeley, CA 94720, U.S.A.}
\email{tataru@math.berkeley.edu}

\thanks{I.B. was supported in part by NSF grant DMS-1001676. J. K. was partially supported by NSF grant DMS-0757278.  D.T. was supported in part  by NSF grant  DMS0354539, as well as by the Miller Foundation}

\begin{abstract} We consider finite energy equivariant solutions for
  the wave map problem from $\R^{2+1}$ to $\S^2$ which are close to the
  soliton family.  We prove asymptotic orbital stability for a
  codimension two class of initial data which is small with respect to
  a stronger topology than the energy.
\end{abstract}

\maketitle

\section{Introduction}
\label{sec:intro}

We consider Wave Maps $U:\R^{2+1}\rightarrow \S^{2}$ which are
equivariant with co-rotation index~$1$. In particular, they satisfy
$U(t,\omega x)=\omega U(t,x)$ for $\omega\in SO(2,\R)$, where the
latter group acts in standard fashion on $\R^{2}$, and the action on
$\S^{2}$ is induced from that on $\R^{2}$ via stereographic projection.
Wave maps are characterized by being critical with respect to the
functional
\begin{equation}\nonumber
  U\rightarrow \int_{\R^{2+1}}\la \partial_{\alpha}U, \partial^{\alpha}U\ra \,d\sigma,
\qquad \,\alpha=0,1,2
\end{equation}
with Einstein's summation convention being in force,
$\partial^{\alpha}=m^{\alpha\beta}\partial_{\beta}$,
$m_{\alpha\beta}=(m^{\alpha\beta})^{-1}$ the Minkowski metric on
$\R^{2+1}$, and $\,d\sigma$ the associated volume element. Also,
$\la\cdot,\cdot\ra$ refers to the standard inner product on $\R^{3}$
if we use ambient coordinates to describe $u$, $\partial_{\alpha}u$
etc. Recall that the energy is preserved:
\[
\calE(u) = \frac12 \int_{\R^2} \la DU(\cdot,t),DU(\cdot,t)\ra\, dx =\const
\]
The problem at hand is {\it{energy critical}}, meaning that the
conserved energy is invariant under the natural re-scaling
$U\rightarrow U(\lambda t, \lambda x)$.

We focus on a particular subset of equivariant maps characterized by the additional
property that $U(t,r,\theta)=(u(t,r),\theta)$ in spherical coordinates, where, on the right-hand side,
$u$ stand for the longitudinal angle and $\theta$ stands for the latitudinal angle, while,
on the left-hand side, $r,\theta$ are the polar coordinates on
$\R^{2}$. Now $u(t,r)$, a scalar function, satisfies
the equation
\begin{equation}\label{maineq}
  -u_{tt}+u_{rr}+\frac{u_{r}}{r}=\frac{\sin(2u)}{2r^{2}}
\end{equation}
Then the energy has the form
\begin{equation} \label{E}
\calE(u) = \pi \int_{\R^2} (|u_t|^2 + |u_r|^2 + \frac{\sin^2(u)}{r^{2}}) rdr
\end{equation}
We shall be interested in co-rotational maps that are topologically
non-trivial,namely with
\[
u(t,0)= 0, \qquad u(t, \infty) = \pi. 
\]
A natural space adapted to the elliptic part of this energy is $\dHe$
\[
\| f \|^2_{\dHe} = \| \partial_r f \|_{L^2}^2 + \| \frac{f}r \|_{L^2}^2
\]
This is the equivariant translation of the usual two dimensional space $\dot H^1$.
The size of the elliptic part of the energy of $u$ in \eqref{E} and its $\dHe$  norm
are comparable provided that $u$ is small pointwise. This is not true directly
for $u$ but it true after we subtract from $u$ the "nearby" soliton which we describe below.

The solitons for \eqref{maineq} have the form
\[
Q_\lambda (r)= Q(\lambda r), \qquad Q(r) = 2\arctan r, \qquad \lambda \in \R_+=(0,\infty)
\]
and are global minimizers of the energy $\calE$ within their homotopy 
class, $\calE(Q_\lambda)=4\pi$. 

We consider solutions $u$ which are close to the soliton in the sense 
that
\begin{equation} \label{unearQ}
\calE(u) - \calE(Q) \ll 1
\end{equation}
As it turns out, such solutions must stay close to the soliton 
family $\{ Q_\lambda \}$ due to the bound
\begin{equation}\label{minlambda}
\inf_\lambda \|(u(t)-Q_\lambda)\|_{\dHe}^2 + \|u_t(t)\|_{L^2}^2
\sim \calE(u) - \calE(Q)
\end{equation}
Indeed, this follows for example from \cite{Co}. 
Thus at any given $t$ one can choose some $\lambda(t)$ so that
\begin{equation}\label{goodlambda}
 \| (u(t)-Q_\lambda)\|_{\dHe}^2 + \|u_t(t)\|_{L^2}^2
\sim \calE(u) - \calE(Q)
\end{equation}
Such a parameter $\lambda$ is uniquely determined up to an error of
size $O((\calE(u) - \calE(Q))^\frac12)$. One can for instance choose
$\lambda$ to be the minimizer in \eqref{minlambda} though there are no
obvious benefits to be derived from that. Another equivalent choice 
is more direct, namely by the relation
\begin{equation}
u(t,\lambda^{-1}(t)) = \frac{\pi}2
\end{equation}
and this still satisfies \eqref{goodlambda}, see for instance \cite{BT}.
Since this problem is locally well-posed in the energy space,
scaling considerations show that (for well chosen $\lambda(t)$)
we have 
\begin{equation}\label{lambdas}
\left|\frac{d}{dt} \lambda(t) \right| \lesssim \lambda^{-2} 
\end{equation}
so at least locally $\lambda$ stays bounded.
Then the main question to ask is as follows:

\begin{open}\nonumber
What is the behavior of the function $\lambda(t)$ for equivariant maps  
 satisfying \eqref{unearQ} ?
\end{open}
We can distinguish several interesting plausible scenarios:

\begin{itemize}
\item{\em Type 1:} $\lambda(t) \to \infty$ as $t \to t_0$ (finite time
  blow-up). By \eqref{lambdas} this can only happen at rates
  $\lambda(t) \gtrsim |t-t_0|^{-1}$. The above extreme corresponds to
  self-similar concentration; this can be thought also as a
  consequence of the finite speed of propagation. In effect, by the
  important work \cite{Str2}, it is known that such blow up can only
  occur with speed strictly faster than self-similar:
\[
\lambda(t)|t-t_0|\rightarrow\infty
\]

\item{\em Type 2:} $\lambda(t) \to \infty$ as $t \to \infty$ (infinite
  time focusing).

\item{\em Type 3:} $\lambda(t) \to 0$ as $t \to \infty$ ( infinite
  time relaxation).  By \eqref{lambdas} this can only happen at rates
  $\lambda(t) \gtrsim t^{-1}$, which corresponds to self-similar
  relaxation.

\item{\em Type 4:} $\lambda(t)$ stays in a compact set globally in
  time.  Then we have a global solution, and possibly a resolution
  into a soliton plus a dispersive part.

\end{itemize}

Blow-up solutions of Type 1 were constructed not long ago in two quite
different papers \cite{KST} and \cite{Ro-St}, and the result of the
latter paper was significantly strengthened and generalized in
\cite{Ra-Ro}. The behavior of $\lambda(t)$ in \cite{KST} as $t \to 0$
is given by
\[
\lambda(t) = t^{-1-\nu}, \qquad \nu \geq  1
\]
(here the restriction $\nu \geq 1$ seems technical, should really be $\nu > 0$)
while that in  \cite{Ra-Ro} is 
\[
\lambda(t) \sim t^{-1} e^{c \sqrt{\log t}}
\]
The latter solutions were also proved to be stable with respect to
class of small smooth perturbations. It is not implausible that the
set of all blow-up solutions is open in a suitable topology, although
numerical evidence in \cite{Bi} appears to suggest the existence of a
co-dimension one manifold of data leading to an unstable blow up,
which separates scattering solutions from a stable regime of finite
time blow up solutions.

Up to this point we are not aware of any examples of solutions of type
2, 3 and also of type 4 other than the $Q_\lambda$'s in the wave maps
context, although recent work \cite{Gu-Na-Tsai} revealed unusual
solutions of this type in the context of the Landau-Lifshitz
equation. Earlier work \cite{KS2} showed the existence of type 4
solutions for the critical focusing nonlinear wave equation on
$\R^{3+1}$.

Understanding the general picture for data in the energy space 
seems out of reach for now. However, there is a simpler question one may 
ask, namely what happens for data which is close to a soliton in a stronger
topology, which includes both extra regularity and extra decay at infinity.
Neither the results of \cite{KST} nor the ones in \cite{Ra-Ro} apply 
in this context. A good starting point for this investigation is the following

\begin{conj}
There exists a codimension one set of (small) data leading to Type 4 solutions, 
which separates Type 1 and Type 3 solutions.
\end{conj}

One should take this only as a rough guide; some fine adjustments may
be needed. Our main result is to construct a large class of Type 4
solutions:

\begin{theorem}
There exists a codimension two set of Type 4 equivariant wave maps satisfying
 \eqref{unearQ}.
\end{theorem} 

For a more precise formulation of the theorem we refer the reader to
Section~\ref{s:maint}. Compared with the conjecture above, one can see
that we are one dimension short. At this point it is not clear if this
is a technical issue, or something new happens. A plausible scenario
might be that the missing dimension may include Type 2 solutions, as
well as slowly relaxing Type 4 solutions.

One should also compare this result with the related problem for
Schr\"odinger maps. Although the solitons are the same and the
operator $H$ arising below in the linearization is also the same for
Schr\"odinger maps, in \cite{BT} it is shown that the solitons are
stable with respect to small localized perturbations. One way to 
explain this is that the linear growth in the resonant direction
occuring in the $H$-wave equation has a stronger destabilizing effect 
than the corresponding lack of decay in the $H$-Schr\"odinger  equation.

\subsection{Notations}
Here we introduce a few notation which will be used throughout the paper.
We slightly modify the use of $\la \cdot \ra$ in the following sense
\[
\la x \ra=\sqrt{4+x^2}, \qquad x \in \R
\]
For a real number $a$ we define $a^+=\max\{ 0, a \}$ and $a^-=\min\{0,a\}$.

We will use a dyadic partition of $\R_+$ into sets  $\{A_m\}_{m \in \Z}$ given by
\[
A_m = \{  2^{m-1} < r < 2^{m+1}\}.
\]

For given $M > 0$, we use smooth localization functions $\chi_{\les M}, \chi_{\ges M}$ 
forming a partition of unity for $\R_+$ and such that
\[
|(r\partial_r)^\alpha \chi_{\les M}| + |(r\partial_r)^\alpha \chi_{\ges M}| \les_\alpha 1
\]

\section{ The gauge derivative and linearizations} \label{gauge}

The linearized equation \eqref{maineq} around the soliton $Q$ has the form
\begin{equation} \label{mainv}
-v_{tt}-H v=0, \qquad H= - \partial_r^2 - \frac{1}r \partial_r + \frac{\cos(2Q)}{r^{2}}
\end{equation}
The elliptic operator $H$ admits the factorization
\begin{equation} \label{HL}
H=L^*L, \qquad 
L = h_1 \partial_r h_1^{-1} = \partial_r + \frac{h_3}{r}, \qquad
L^{*}=- h_1^{-1} \partial_r h_1 -\frac1{r}= -\partial_r +
\frac{h_3-1}{r}.
\end{equation}
where\footnote{throughout this paper we use
$\sin{Q}, \cos{Q}$ instead of $h_1,h_3$;  however the reader may need this 
correspondence in order to relate to  the work \cite{BT}.}
$h_1=\sin Q=\frac{2r}{1+r^2}, h_3 = -\cos{Q}=\frac{r^2-1}{r^2+1}$. 
$H$ is nonnegative and has a zero resonance 
\[
\phi_0= h_1 = \frac{2r}{1+r^2}
\]
This resonance is the reason why \eqref{mainv} does not have good
dispersive estimates.  Since $\phi_0$ fails to be an eigenvalue, we
cannot project it away as it is usually done in standard modulation
theory. This suggests that working with the variable $u$ and its
equation \eqref{maineq} runs into problems due to the lack of good
linear estimates needed to treat the nonlinearity. Therefore, instead
of working with the solution $u$ we introduce a new variable
\begin{equation}
w = \partial_r u - \frac{1}{r} \sin u
\label{wdef}\end{equation}
which has the nice property that 
\[
 w=0   \Longleftrightarrow u = Q_\lambda 
\]
for some $\lambda \in \R_+$. Indeed, by rearranging \eqref{E} and using $u(0)=0$, $u(\infty)=\pi$, 
we obtain 
\[
\calE(u) = \pi \int_{0}^\infty (|u_t|^2 + |w|^2) rdr +  \pi \int_0^\infty 2 \sin{u} \cdot \partial_r u dr 
= \pi \int_{0}^\infty (|u_t|^2 + |w|^2) rdr + 4 \pi
\]
from which the above observation follows. This type of change of variables originates at least with the work \cite{Gus-Kang-Tsai}. 
If $\lambda(t)$ is chosen such that \eqref{goodlambda} holds, then using \eqref{unearQ}, a direct computation shows that
\begin{equation}
\| u-Q_{\lambda} \|_{\dHe} \approx \| w \|_{L^2}^2.
\end{equation}
Then a direct computation shows that $w$ solves
\begin{equation}
w_{tt} - \Delta w + \frac{2(1+\cos u)}{r^2} w = \frac{1}{r} \sin u (u_t^2 -w^2)
\label{weq}\end{equation}
The function $u$ appears in this equation, but 
it can be recovered from $w$ by solving the  ode \eqref{wdef}
with $Q$-like ``data'' at $r = \infty$. 

We remark that the linearized form of \eqref{wdef} near $Q$
is 
\begin{equation}
z = (\partial_r  - \frac{1}{r} \cos Q) v = L v
\label{wtoulin}\end{equation}
where $L$ was introduced above in \eqref{HL}. 

On the other hand the linearized equation for $w$ near $Q$ has the form
\begin{equation}
z_{tt} - \Delta z + \frac{2(1+\cos Q)}{r^2} z = 0
\label{zeq}\end{equation}
This wave equation is governed by the operator 
\[
\tilde H =  -\Delta + \frac{2(1+\cos Q)}{r^2} =  -\Delta + \frac{4}{r^2(1+r^2)}= LL^*
\]
This operator is better behaved compared to $H$, in particular
its zero mode $\psi_0$ grows logarithmically at infinity.

The plan is to treat the equation \eqref{weq} in a perturbative manner
for the most part. To fix things, we will rewrite it
in the form
\begin{equation}
(\partial_t^2 + \tilde H)  w =  \frac{2(\cos Q- \cos u)}{r^2} w +
 \frac{1}{r} \sin u (u_t^2 -w^2):= N(w,u) 
\label{w1}\end{equation}
and work with this from here on.  The equation
\eqref{w1} for $w$ is preferable due to the nice dispersive properties
of its linear part. However, as $u$ occurs in the $w$ equation, one
has to also keep track of it through the elliptic equation
\eqref{wdef}.

In order to study this equation we
need to understand better the structure of its linear part, and, in
particular, the spectral theory for the operator $\tilde H$. This is
the subject of section \ref{s:ft}.

\subsection{Setup of the problem} \label{s:maint}

The starting point is to consider $\bw$ to be an exact real solution to the 
linear homogeneous equation
\begin{equation}\label{bw}
(\partial_t^2 + \tilde H) \bw = 0, \qquad w(0) = w_0,  \qquad w_t(0) = w_1
\end{equation}
where $w_0$ and $w_1$ are real Schwartz functions which are
assumed to satisfy the nonresonance conditions
\begin{equation}\label{nonres}
\la w_0, \psi_0 \ra = 0, \qquad \la w_1, \psi_0 \ra = 0
\end{equation}
We denote by $\bu$ the corresponding  map, see \eqref{wdef} (this will be made precise in Proposition~\ref{lul}),
obtained by solving the ode
\begin{equation} \label{budef}
\partial_r \bu - \frac{1}{r} \sin \bu = \bw, \qquad \bu \sim Q \ 
\text{as} \ r \to \infty
\end{equation}
Now we seek a solution to the nonlinear equation $u$ and its
associated gauge derivative $w$ close to $\bu,\bw$ respectively,
\begin{equation} \label{egdef}
u = \bu + \tu, \qquad w =\bw + \tw
\end{equation}
so that $u$ and $w$ match $\bu$ and $\bw$  asymptotically as $t \to \infty$.

By a slight abuse of notation we use $\| \cdot \|_S$ to denote a norm 
obtained by adding sufficiently many seminorms of the Schwartz space $S$.
We also use $\lesss$ for inequalities where the implicit constant depends 
on $\|(w_0,w_1)\|_{S}$. Modulo defining the $X$ and $LX$ norms, 
we are now in a position to restate our main result 
in a more detailed fashion.

\begin{theorem} \label{t:main}
  Let $w_0$, $w_1$ be  Schwartz functions  satisfy the
  nonresonance conditions \eqref{nonres}.  Let $\bu$ and $\bw$ be
  defined as above. Then there exists $T \lesss 1$ and a unique wave
  map $u$ in $[T,\infty)$ so that $u$ and $w$ match $\bu$ and $\bw$ as
  $t \to \infty$ in the following asymptotic fashion  for $t \in [T,\infty)$:
\begin{equation} \label{gboot}
  \| \tw(t)\|_{LX} \lesss  t^{-\frac32}, \qquad 
\| \partial_t \tw(t)\|_{LX} \lesss  t^{-\frac52},  \qquad 
\| \tw(t)\|_{\dot H^1} \lesss t^{-\frac52}
\end{equation}
respectively
\begin{equation} \label{eboot}
  \| \tu(t)\|_{X} \lesss t^{-\frac32}, \qquad 
\| \partial_t \tu(t)\|_{LX} \lesss t^{-\frac52}
\end{equation}
Furthermore, the map $u$ and its  corresponding gauge derivative $w$ 
have a Lipschitz dependence on $(w_0,w_1)$ with respect to the above norms. 
\end{theorem}

One would expect the above result to be in terms of $L^2$ and $\dHe$
spaces.  However these spaces are very disconnected from the spectral
structure of $H$ and $\tilde H$, particularly at low frequencies, and
this makes them unsuitable. The spaces $X \subset \dHe$ and $LX
\subset L^2$ have been introduced in \cite{BT} to address exactly this
issue: they are low frequency corrections of $\dHe$, respectively
$L^2$. Their exact definition is provided in the next section.

In view of the equation \eqref{w1}, the function $\gamma$ solves
\begin{equation} \label{geq}
(\partial_t^2 + \tilde H) \gamma =N(\bw + \tw,\bu+ \tu)
\end{equation}
with zero Cauchy data at infinity.
By \eqref{wdef}, \eqref{egdef} and \eqref{budef}, the function $\epsilon$ is determined 
from the equation
\begin{equation} \label{ge}
\tw = \partial_r \tu  - \frac{\sin(\tu + \bu) - \sin \bu}r 
\end{equation}

We proceed as follows. In the next section we recall from \cite{BT} the spectral theory for $H$ (which in fact originates
in \cite{KST}) and $\tilde H$
and the definitions and some properties of the spaces $X$ and $LX$. Then, in Section \ref{s:lin}
we provide linear estimates for the linear (inhomogenous) wave equation corresponding to \eqref{bw}. 
In Section \ref{s:bubw} we analyze the first approximations $\bw$ and $\bu$ using \eqref{budef}.
Then, in Section \ref{s:transition}, we continue with the study of
the relation between $\epsilon$ and $\ga$ based on the the equation \eqref{ge}. 
All the analysis carried in Sections \ref{s:lin}-\ref{s:transition} is done in the context of $X$ and $LX$ spaces.
In the end, in Section \ref{s:nonl} we study the solvability of equation \eqref{geq} using perturbative methods 
in $LX$ based spaces.

\section{The modified Fourier transform}
\label{s:ft}
In this section we recall the spectral theory associated with the
operators $H, \tilde H$.  The spectral theory for $H$ was developed in
\cite{KST} and the one for $\tilde H$ was derived from the one for $H$
in \cite{BT}. In this paper, we follow closely the exposition in
\cite{BT}. 

\subsection{Generalized eigenfunctions}
We consider $H$ acting as an unbounded selfadjoint operator
in $L^2(rdr)$. Then $H$ is nonnegative, and its spectrum $[0,\infty)$
is absolutely continuous. $H$ has a zero resonance, namely
$\phi_0=h_1$,
\[
 H h_1 = 0.
\]
For each $\xi > 0$ one can choose a normalized generalized eigenfunction
$\phi_\xi$,
\[
 H \phi_\xi = \xi^2 \phi_\xi. 
\]
These are unique up to a $\xi$ dependent multiplicative factor,
which is chosen as described below.

To these one associates a generalized Fourier transform $\FH$ defined by
\[
 \FH{f}(\xi)=\int_0^\infty \phi_\xi(r) f(r)  rdr
\]
where the integral above is considered in the singular sense.
This is an $L^2$ isometry, and we have the inversion formula
\[
 f(r) = \int_0^\infty \phi_\xi(r) \FH{f}(\xi)  d\xi 
\]
The functions $\phi_\xi$ are smooth with respect to both $r$ and $\xi$.
To describe them one considers two distinct regions, $r \xi \lesssim 1$ 
and $r \xi \gtrsim  1$.

In the first region $r\xi \lesssim 1$ the functions $\phi_\xi$
admit a power series expansion of the form
\begin{equation} \label{repphi}
\phi_\xi (r)= q(\xi) \left( \phi_0 + \frac{1}{r} \sum_{j=1}^\infty (r\xi)^{2j}
\phi_j(r^2)\right), \qquad r\xi \lesssim 1
\end{equation}
where $\phi_0=h_1$ and the functions $\phi_j$ are analytic and satisfy 
\begin{equation} \label{derphi}
|(r \partial_r)^\alpha \phi_j| \lesssim_\alpha \frac{C^j}{(j-1)!} \log{(1+r)}
\end{equation} 
This bound is not
spelled out in \cite{KST}, but it follows directly from the integral
recurrence formula for $f_j$'s (page 578 in the paper).
The smooth positive weight $q$ satisfies 
\begin{equation}\label{qest}
 q(\xi) \approx \left\{ \begin{array}{ll}
\displaystyle \frac{1}{\xi^\frac12 |\log \xi|},  &  \xi \ll 1 \cr\cr
\xi^{\frac32},   &  \xi \gg 1
                        \end{array} \right., \qquad
|(\xi \partial_\xi)^\alpha q| \lesssim_\alpha q
\end{equation}
Defining the weight 
\begin{equation}\label{defmk1}
 m_k^1(r)=
\left\{
\begin{array}{ll}
 \min\{1, r 2^{k} \dfrac{\ln{(1+r^2)}}{\langle k \rangle}\},  & \ k < 0 \\ \\
 \min\{1, r^3 2^{3k}\},  & \ k \geq 0
\end{array}
\right.
\end{equation}
it follows that the nonresonant part of $\phi_\xi$ satisfies
\begin{equation}\label{pointphilow}
|(\xi \partial_\xi)^\alpha (r \partial_r)^\beta \left( \phi_\xi(r) - q(\xi) \phi_0(r)\right)|
\lesssim_{\alpha\beta} 2^{\frac{k}2} m_k^1(r), \qquad \xi \approx 2^k,\  r\xi \lesssim 1 
\end{equation}

In the other region $r \xi \gtrsim  1$  we begin with the functions
\begin{equation} \label{repphi+}
\phi^{+}_\xi(r)= r^{-\frac12} e^{ir\xi} \sigma(r\xi,r), 
\qquad r\xi \gtrsim  1
\end{equation}
solving 
\[
H \phi^{+}_\xi = \xi^2 \phi^+_\xi
\]
where for $\sigma$ we have the following asymptotic expansion
\[
\sigma(q,r) \approx \sum_{j=0}^\infty q^{-j} \phi^{+}_j(r), 
\qquad \phi_0^{+}=1 , \qquad \phi_1^{+}=\frac{3i}{8} + O(\frac1{1+r^2})
\]
with
\[
\sup_{r > 0} |(r\partial_r)^k \phi^{+}_j| < \infty
\]
in the following sense
\[
\sup_{r > 0} | (r \partial r)^\alpha (q \partial_q)^\beta [\sigma(q,r)-\sum_{j=0}^{j_0}
q^{-j} \phi^{+}_{j}(r) ] | \leq c_{\alpha,\beta,j_0} q^{-j_0-1}
\]
Then we have the representation
\begin{equation} \label{phipsi}
\phi_{\xi}(r)=a(\xi) \phi^{+}_\xi(r) + \overline{a(\xi) \phi^{+}_\xi(r)}
\end{equation}
where the complex valued function $a$ satisfies
\begin{equation} \label{abound}
|a(\xi)| = \sqrt{\frac2{\pi}}, \qquad | (\xi \partial_\xi)^\alpha a(\xi)| \lesssim_\alpha 1
\end{equation}

The spectral theory for $\tilde H$ is derived from the spectral theory 
for $H$ due to the conjugate representations
\[
 H = L^* L, \qquad \tilde H = L L^*
\]
This allows us to define generalized eigenfunctions $\psi_\xi$ for
$\tilde H$ using the generalized eigenfunctions $\phi_\xi$ for
$ H$,
\begin{equation} \label{phipsirel}
 \psi_\xi = \xi^{-1} L \phi_\xi, \qquad L^* \psi_\xi = \xi \phi_\xi
\end{equation}
It is easy to see that $\psi_\xi$ are real, smooth, vanish at $r = 0$
and solve
\[
 \tilde H \psi_\xi = \xi^2 \psi_\xi
\]
With respect to this frame we can define the generalized Fourier transform
adapted to $\tilde H$ by 
\[
 \FtH{f}(\xi)=\int_0^\infty \psi_\xi(r) f(r)  rdr
\]
where the integral above is considered in the singular sense.
This is an $L^2$ isometry, and we have the inversion formula
\begin{equation} \label{FTL0}
 f(r) = \int_0^\infty \psi_\xi(r) \FtH{f}(\xi)  d\xi 
\end{equation}
To see this we compute, for a Schwartz function $f$:  
\[
\begin{split}
 \FtH{Lf}(\xi) & =\! \int_0^\infty  \psi_\xi(r) L f(r)  rdr
= \!\int_0^\infty L^* \psi_\xi(r)  f(r)  rdr
\\ &= \!\int_0^\infty \xi \phi_\xi(r)  f(r)  rdr = \xi \FH{f}(\xi) 
\end{split}
\]
Hence
\[
 \| \FtH{Lf}\|_{L^2}^2 = \| \xi \FH{f}(\xi)\|_{L^2}^2 = \la H f,f\ra_{L^2(rdr)}
= \|Lf\|_{L^2}^2
\]
which suffices since $Lf$ spans a dense subset of $L^2$.

The representation of $\psi_\xi$ in the two regions $r\xi \lesssim 1$
and $r\xi \gtrsim 1$ is obtained from the similar representation of $\phi_\xi$. 
In the first region $r\xi \lesssim 1$ the functions $\psi_\xi$
admit a power series expansion of the form
\begin{equation}\label{psirepa}
\psi_\xi = \xi q(\xi) \left(\psi_0(r) +   \sum_{j \geq 1} (r\xi)^{2j} 
{\psi}_j(r^2)\right)
\end{equation}
where
\[
{\psi}_j(r)= ( h_3+1 +2j) \phi_{j+1}(r) + r \partial_r \phi_{j+1}(r)
\]
From \eqref{derphi}, it follows that 
\[
 |(r \partial_r)^\alpha \psi_j| \lesssim_\alpha \frac{C^j}{(j-1)!} \log{(1+r^2)}
\]
In addition, $\psi_0$ solves $L^*\psi_0 = \phi_0$  therefore
a direct computation shows that
\[
 \psi_0 = \frac1{2} \left(\frac{(1+r^2)\log(1+r^2)}{r^2}-1 \right)
\]

In particular, defining the weights 
\begin{equation}\label{defmk}
m_k(r)=
\left\{
\begin{array}{ll}
\min\{1, \dfrac{\ln{(1+r^2)}}{\langle k \rangle}\}, & k < 0 \\ \\
\min\{1, r^2 2^{2k}\}, & k \geq 0
\end{array}
\right.
\end{equation}
we have the pointwise bound for $\psi_\xi$ 
\begin{equation} \label{pointtp}
| (r \partial_r)^\alpha (\xi \partial_\xi)^\beta \psi_{\xi}(r) | \lesssim_{\alpha\beta}
2^{\frac{k}2} m_k(r) , \qquad \xi \approx 2^k,\  r\xi \lesssim 1 
\end{equation}

On the other hand in the regime $r \xi \gtrsim  1$ we define 
\[
 \psi^+ = \xi^{-1} L\phi^+ 
\]
and we obtain the representation
\begin{equation} \label{psirep}
\psi_{\xi}(r)=a(\xi) \psi^{+}_\xi(r) + \overline{a(\xi) \psi^{+}_\xi(r)}
\end{equation}
For $\psi^+$ we obtain the expression
\begin{equation} \label{reppsi}
\psi^{+}_\xi(r)= r^{-\frac12} e^{ir\xi} \tilde\sigma(r\xi,r), 
\qquad r\xi \gtrsim  1
\end{equation}
where $\tilde \sigma$ has the form
\[
\tilde\sigma(q,r)  = i \sigma(q,r) -\frac12 q^{-1} \sigma(q,r)
+ \frac{\partial}{\partial q} \sigma(q,r)+ \xi^{-1} L \sigma(q,r) 
\]
therefore it has exactly the same properties as $\sigma$. In particular,
for fixed $\xi$, we obtain that
\begin{equation}
\tilde{\sigma}(r\xi,r) = i -\frac78  r^{-1}\xi^{-1} + O(r^{-2})
\end{equation}

We conclude our description of the generalized eigenfunctions and of
the associated Fourier transforms with a bound on the $\tilde H$ Fourier 
transforms of Schwartz functions.

\begin{lemma} If $f$ is a Schwartz function satisfying $\la f, \psi_0 \ra=0$ then
\begin{equation} \label{nonresbound}
|(\xi \partial_\xi)^\alpha \FtH f(\xi)| \les_{\alpha,N}
\left\{ 
\begin{array}{ll}
 \frac{\xi^\frac52}{\la \log \xi\ra}, \quad  & \xi \les 1 \cr
\la \xi \ra^{-N}, \quad &  \xi \ges 1
\end{array}
\right.
\end{equation}
\end{lemma}

\begin{proof} We start from the definition of modified Fourier transform and use that $\la f, \psi_0 \ra=0$
\[
\begin{split}
| \FtH f(\xi)| & \les \left( |\int_0^{\xi^{-1}} \psi_\xi(r) f(r) r dr | + |\int_{\xi^{-1}}^\infty \psi_\xi(r) f(r) r dr | \right) \\
& \les \xi q(\xi) \left( \int_{\xi^{-1}}^\infty |\psi_0(r) f(r)| rdr + \int_0^{\xi^{-1}} \sum_{j \geq 1} (r \xi)^{2j} \psi_j(r^2) f(r) rdr\right) + 
\int_{\xi^{-1}}^\infty |f(r)| r^\frac12 dr  \\
& \les \xi^3 q(\xi)
\end{split}
\]
A similar argument takes care of the case $\alpha > 0$. 

\end{proof}

\subsection{The spaces $X$ and $LX$}
The operator $L$ maps $\dHe$ into $L^2$. Conversely one would like that, given some $f \in L^2$, we could solve 
$Lu = f$ and we obtain a solution $u$ which is in $\dHe$ and satisfies
\[
\| u \|_{\dHe} \lesssim \|f\|_{L^2}
\]
However, this is not the case. The first observation is that
the solution is only unique modulo a multiple of the
resonance $\phi_0$. Moreover the inequality above is not expected to be true, even
assuming that somehow we choose the "best" $u$ from all candidates.

The spaces $X$ and $LX$ are in part introduced in order to remedy both the 
ambiguity in the inversion of $L$ and the failing inequality.

\begin{defi} 
a) The space $X$ is defined as the completion of the subspace 
of $L^2(rdr)$ for which the following norm is finite  
\[
\| u \|_{X} = \left( \sum_{k \geq 0} 2^{2k} \| P_k^H u \|_{L^2}^2 \right)^\frac12
+ \sum_{k < 0} \frac1{|k|} \| P_k^H u \|_{L^2}
\]
where $P^H_k$ is the  Littlewood-Paley operator localizing at
frequency $\xi \approx 2^k$ in the $H$ calculus.

b) $L X$ is the space of functions of the form $f=L
u$ with $u \in X$, with norm $\| f \|_{L X}= \| u \|_{X}$.
 Expressed in the $\tilde H$ calculus, the $LX$ norm is written as
\[
\| f \|_{LX} = \left( \sum_{k \geq 0}  \| P_k^{\tilde H} f \|_{L^2}^2
\right)^\frac12 + \sum_{k < 0} \frac{2^{-k}}{|k|} \| P_k^{\tilde H} f \|_{L^2}
\]
\end{defi}

In this article we work with equivariant wave maps $u$ for which $ \| u - Q\|_{X} \ll 1$. 
This corresponds to functions $w$ which satisfy $ \| w \|_{LX} \ll 1$.  
The simplest properties of the space
$X$ are summarized as follows, see Proposition 4.2 in \cite{BT}:

\begin{prop}
The following embeddings hold for the space $X$:   
\begin{equation} \label{Xembt}
\He \subset  X \subset \dHe
\end{equation}
In addition for $f$ in $X$ we have the following bounds:
\begin{equation} \label{pointX}
\|\langle r \rangle^\frac12 f \|_{L^\infty} \lesssim  \| f \|_{X} 
\end{equation}
\begin{equation} \label{linX}
\left\| \frac{f}{\ln (1+r)}  \right\|_{L^2} \lesssim  \| f \|_{X}
\end{equation}
\begin{equation} \label{linX4}
\left\| \langle r \rangle^\frac12{f} \right\|_{L^4} \lesssim  \| f \|_{X}
\end{equation}

\end{prop}

Now we turn our attention to the space $LX$. From \cite{BT}, Lemma 4.4 and Proposition 4.5, we have 
\begin{lemma} If $f \in L^2$ is localized at $\tilde{H}$- frequency $2^k$ then
\begin{equation} \label{ps0}
|  f(r) | \lesssim 2^k m_k(r)(1+2^k r)^{-\frac12}  \| f \|_{L^2}
\end{equation}

\end{lemma}

\begin{proposition}
 The following embeddings hold for $LX$:
\begin{equation}
L^1 \cap L^2 \subset LX \subset L^2
\label{LXemb}\end{equation}
\end{proposition}

\section {Linear estimates for the $\tilde H$ wave equation}
\label{s:lin}
In this section we prove estimates for the linear equation
\begin{equation} \label{lineq}
(\partial_t^2 + \tilde H) \psi = f
\end{equation}
with zero Cauchy data at infinity. The solution  is given by $\psi=Kf$, where
\[
K f(r,t)= - \FtH^{-1} \int_t^\infty \frac{\sin(t-s)\xi}{\xi} \FtH f(\xi,s) ds
\]
We also need its time derivative, which is given by
\[
\partial_t K f = - \FtH^{-1} \int_t^\infty \cos(t-s)\xi \cdot \FtH f(\xi,s) ds
\]
Finally we need the following formula, which follows from \eqref{phipsirel}
\[
L^* Kf = - \FH^{-1} \int_t^\infty \sin(t-s)\xi \cdot \FtH f(\xi,s) ds
\]
The following result is a modification of the standard energy estimate
for the wave equation:

\begin{lemma} 
\label{l:Kest}
Assume that $f(s) \in LX$. Then for every $\alpha > 0 $, the solution
of \eqref{lineq} with zero data at $\infty$ satisfies
\begin{equation} \label{Kest}
t^\alpha \| \psi(t) \|_{LX} + t^{\alpha+1} 
( \| \partial_t \psi(t) \|_{LX} + \| \psi(t) \|_{\dHe} ) \les  \sup_s s^{\alpha + 2} \| f(s) \|_{LX}
\end{equation}

\end{lemma}

\begin{proof} The solution of \eqref{lineq} with zero data at $\infty$
  is given by $\psi=Kf$.  The estimate for the first term follows from
  the bound $|\frac{\sin(t-s)\xi}{\xi}| \les |t-s|$ and the representation of the spaces $LX$ 
on the Fourier side. The estimate   for the second term is similar.

The argument for the third term is more involved. We denote 
\[
\mathcal{F}_{\tilde H} g(t,\xi)=- \int_t^\infty \sin((t-s)\xi) \FtH f(\xi,s) ds
\]
Then 
\[
\xi\mathcal{F}_{\tilde H} \psi(t,\xi) =  \mathcal{F}_{\tilde H} g(t,\xi)
\]
We estimate as above
\[
\|  g(t) \|_{LX} \les \int_t^\infty \| f(s) \|_{LX} ds \lesssim t^{-\alpha -1}
 \sup_s s^{\alpha + 2} \| f(s) \|_{LX}
\]
Hence it suffices to show that for $\psi$ and $g$ related as above
we have
\begin{equation}
\|\psi\|_{\dHe} \lesssim  \|g\|_{LX}
\end{equation}
Here the time variable plays no role and is discarded.  Recalling the
form of $L^*$ from \eqref{HL}, namely $L^{*}= -\partial_r +
\frac{h_3-1}{r}$, it follows that
\[
\| \psi \|_{\dHe} \les \| L^* \psi \|_{L^2} + \| \frac{\psi}{r} \|_{L^2}
\]
For the first term we use Plancherel to write
\[
\| L^*   \psi(t) \|_{L^2}^2 = \la \psi(t), \tilde H \psi(t)\ra
= \| \xi \mathcal{F}_{\tilde H} \psi(\xi)\|_{L^2}^2 = 
\| g\|_{L^2} \lesssim \|g\|_{LX}^2
\]
For the second term the $L^2$ bound for $g$ no longer suffices, and 
 we need to use the $LX$ norm of $g$. We consider a Littlewood-Paley 
decomposition for both $\psi$ and $g$, and denote their dyadic pieces by 
$\psi_k$, respectively $g_k$. Then 
\[
\| \psi_k\|_{L^2} \approx 2^{-k} \|g_k\|_{L^2}
\]
By using \eqref{pointtp}, \eqref{psirep} and the Cauchy-Schwartz inequality we obtain
pointwise bounds for $\psi_k$, namely
\[
|\psi_k| \lesssim \frac{m_k(r)}{\langle 2^{k} r \rangle^\frac12} 2^k \| \psi_k\|_{L^2}
\lesssim \frac{m_k(r)}{\langle 2^{k} r \rangle^\frac12}  \| g_k\|_{L^2}
\]
with $m_k$ as in \eqref{defmk}. For $k \geq 0$ the contributions are almost 
orthogonal and we  obtain
\[
\| \frac{\psi_{\geq 0}}r\|_{L^2} \lesssim \| g_{\geq 0}\|_{L^2}
\]
However, if $k < 0$ then the weaker logarithmic decay for small $r$ no longer 
suffices for such an argument.  Instead by direct computation
we obtain a weaker bound, 
\[
\| \frac{\psi_{k}}r\|_{L^2} \lesssim |k|^{\frac12} \| g_k\|_{L^2} 
\lesssim |k|^{\frac32} 2^{k} \|g\|_{LX}
\]
Then the $k$ summation is easily accomplished.

\end{proof}

\section{Analysis of the first approximations  
$\bw$ and $\bu$}
\label{s:bubw}

\subsection{ Pointwise bounds for $\bw$}

We denote $f_0=\FtH w_0$ and $f_1=\FtH w_1$.  Then for $\bw$ we have
the representation
\[
\bw(t,r) = \int_0^\infty  \psi_\xi(r)
(f_0(\xi) \cos(t\xi) + \frac{1}{\xi} f_1(\xi) \sin(t\xi)) d\xi
\]
Since $w_0, w_1$ are Schwartz functions  satisfying 
\eqref{nonres}, from \eqref{nonresbound} we obtain
\begin{equation} \label{f01}
|(\xi \partial_\xi)^\alpha f_0(\xi)| + |(\xi \partial_\xi)^\alpha f_1(\xi)| \les_{\alpha, N} 
 \|(w_0,w_1)\|_{S} \left\{ 
\begin{array}{ll}
 \frac{\xi^\frac52}{\la \log \xi\ra}, \quad  & \xi \les 1 \cr
\la \xi \ra^{-N}, \quad &  \xi \ges 1
\end{array}
\right.
\end{equation}
Here by a slight abuse of notation we use $\|.\|_S$ to denote a finite collection  
of the $S$ seminorms. This will allow us to obtain pointwise bounds for $\bw$:

\begin{lemma} If  $w_0, w_1$ are Schwartz functions  satisfying the moment conditions
\eqref{nonres}  then $\bw$ satisfies
\begin{equation}\label{bwpoint}
|\bw(r,t)| \les \frac{\log(1+r^2)}{\log\la r+t\ra} \frac{1}{\la t+r \ra ^{\frac12} \la t-r \ra^{\frac52} \log\la r-t\ra} \|(w_0,w_1)\|_{S}
\end{equation}
\end{lemma}

\begin{proof} We fix $k$ and consider
\[
\bw_k(t,r)=\int_0^\infty   \psi_\xi(r) (f_0(\xi) \cos(t\xi) + \frac{1}{\xi} f_1(\xi) \sin(t\xi)) \chi_{k}(\xi) d\xi
\]
For $\psi_\xi(r)$ we use the representation \eqref{psirepa} in the region
$\{ r \xi \lesssim 1\}$, respectively  \eqref{psirep} in the region
$\{ r \xi \gtrsim 1\}$. Then via a standard stationary phase 
argument  we obtain
\[
|w_k(r,t)|\les_N \frac{2^\frac{k}2 \la 2^k r \ra^{-\frac12}  m_k(r)}{\la 2^k |r-t| \ra^N \la k^{-} \ra} 2^{\frac{5k}2} 2^{-Nk^+}.
\]
The desired estimate \eqref{bwpoint} follows by summing these bounds with 
respect to $k$.
\end{proof}

\subsection{ Bounds for  $\bu$, $\bu_t$}

Next we consider $\bu$, which is recovered from $\bw$ via \eqref{budef}.
This equation contains a nonlinear part coming from the sine function.
Consequently, we split $\bu$ into a linear and a nonlinear part:
\[
\bu = Q+  \bu^l + \bu^{nl}
\]
where $\bu^l$ solves the linear part of \eqref{budef}
\[
L \bu^l = \bw
\]
and $\bu^{nl}$ solves 
\begin{equation} \label{nleq}
L \bu^{nl} = N(\bu^l, \bu^{nl})
\end{equation}
where
\[
N(u,v) = \frac1r \big[  \sin{Q} \cdot (\cos (u + v)-1) + \cos{Q} \cdot (\sin (u + v) - (u + v)) \big]
\]
Both of the above ode's are taken with zero Cauchy data at infinity or, equivalently,
can be interpreted via the diffeomorphism $L: X \to LX$. 
The linear part $\bu^l$ is recovered from the explicit formula
\[
\bu^l := L^{-1} \bw = \int_0^\infty  \xi^{-1} \phi_\xi(r)
(f_0(\xi) \cos(t\xi) + \frac{1}{\xi} f_1(\xi) \sin(t\xi)) d\xi
\]
and will be split into a resonant and a nonresonant part $ \bu^l =
\bu^{l,r}+ \bu^{l,nr}$.
 
For the nonlinear part we use an iterative argument based on the fact that
there is enough decay on the right-hand side so that we can recover it via
\begin{equation}\label{bunl}
\bu^{nl} = h_1(r) \int_r^\infty \frac{N(\bu^l, \bu^{nl})}{h_1(s)} ds
\end{equation}
At this stage we also want to keep track of the differences of solutions. 
For this we denote by $\delta w_0, \delta w_1, \delta \bw, \delta \bu$ the corresponding 
differences. 

\begin{proposition} \label{lul} a)  Assume that $w_0, w_1$ 
are small Schwartz functions  satisfying \eqref{nonres}. Then 
\begin{equation} \label{bubound}
\bu^{l}=\bu^{l,r} + \bu^{l,nr},
\end{equation}
where $\bu^{l,r}$ and $\bu^{l,nr}$ satisfy the following bounds
\begin{equation} \label{bulr}
\begin{split}
|\bu^{l,r}| + r |\partial_r \bu^{l,r}|+  \la r+ t \ra |\partial_t \bu^{l,r}| \les & \ \frac{h_1(r)}{\la t+r \ra\log^2\la  t +r \ra} \|(w_0,w_1)\|_{S}, 
\\
| \bu^{l,nr}|+  \frac{r\la r-t \ra}{\la t+r \ra}  |\partial_r \bu^{l,nr}|+ \la r-t \ra | \partial_t \bu^{l,nr}| \les & \ 
 \frac{r}{r+\la t \ra} \frac{1}{\la t+r \ra ^{\frac12} \la t-r \ra^{\frac32}\log \la t-r \ra} \|(w_0,w_1)\|_{S}.
\end{split}
\end{equation}
In addition, 
\begin{equation} \label{ulrt}
|(\partial_r + \partial_t) \bu^l +  \frac1{2r} \bu^l| 
 \les \frac{1}{t^\frac52 \la r-t \ra^\frac12 \log \la t-r \ra} \|(w_0,w_1)\|_{S},
\qquad r \sim t
\end{equation}

b) For $t \gtrsim_S 1$ the nonlinear part $\bu^{nl}$ satisfies the bounds
\begin{equation} \label{unl}
\begin{split}
|\bu^{nl}(r,t)| \lesss & \  h_1(r) t^{-1.5} \|(w_0,w_1)\|_{S}, 
\\ |\partial_t \bu^{nl} + \frac16 h_1 (\bu^l)^3| \lesss & \ h_1(r) t^{-2} \|(w_0,w_1)\|_{S}
\end{split}
\end{equation}

c) The above estimates hold true for $\delta \bu^{nl}$ and  $\delta \partial_t \bu_l$,
\begin{equation} \label{dunl}
\begin{split}
|\delta \bu^{nl}(r,t)| \lesss & \  h_1(r) t^{-1.5} \|(\delta w_0,\delta w_1)\|_{S}, 
\\ |\partial_t \delta \bu^{nl} + \frac16 h_1 \delta (\bu^l)^3| \lesss & \ h_1(r) t^{-2} \|(\delta w_0,\delta w_1)\|_{S}
\end{split}
\end{equation}
\end{proposition}

\begin{remark}
  By finite speed of propagation arguments it is not difficult to show
  that $\bu^l$ decays rapidly outside the cone. However, for our
  purposes the decay established in the above proposition suffices.
\end{remark}

\begin{remark}
The bound \eqref{ulrt} shows that a double cancellation occurs on the light cone,
as opposed to the expected single cancellation. This is a consequence of the exact 
decay properties at infinity for the potential in $\tilde H$.
\end{remark}

\begin{remark}  The second estimate in part (b) is the outcome of 
a more subtle nonlinear cancellation, rather then a brute force computation.
\end{remark}


\begin{proof} a) We first split $\bu^l$ into two parts,
\[
\bu^l(r,t)=\sum_{k} \bu_k^l(r,t) = \sum_{2^k \les\  r^{-1}} \bu_k^l(r,t) + \sum_{2^k \ges\ r^{-1}} \bu_k^l(r,t):= \bu^l_{low}(r,t)+ \bu^l_{hi}(r,t)
\]
where
\[
\bu^l_k=\int \xi^{-1} \phi_\xi(r) \chi_k(\xi) \left[ \cos (t\xi) \cdot \hat f_0(\xi) + 
\frac{\sin(t\xi)}{\xi} \hat{f}_1(\xi)\right] d\xi 
\]
Further, using the power series \eqref{repphi}, we can write
\[
\bu^l_k = \int \xi^{-2} q(\xi) \sin(t\xi) (\phi_0(r)+ \frac1r \sum_{j \geq 1} (r\xi)^{2j} \phi_j(r^2) ) \hat{f}_1(\xi) \chi_k(\xi) d\xi, \qquad 2^k r \lesssim 1
\]
which leads to a corresponding decomposition 
\[
\bu^l_{low} = \bu^{l,0}_{low} + \sum_{j \geq 1} \bu^{l,j}_{low}
\]
Then we set 
\begin{equation}
\bu^{l,r} =  \bu^{l,0}_{low}, \qquad \bu^{l,nr} =   
\bu^l_{hi} +  \sum_{j \geq 1} \bu^{l,j}_{low}
\end{equation}
and proceed to estimate all of the above components of $\bu^{l}$.

The terms in $\bu^l_{hi}$ are estimated by stationary phase using 
\eqref{f01} and the $\phi_\xi$ representation in \eqref{phipsi}.
 This yields
\begin{equation}\label{bulhia}
|\bu^{l}_k| \les  \frac{r ^{-\frac12} 2^{\frac{3k}2}}{\la 2^k |r-t| \ra^N   \la k^{-} \ra } 2^{-Nk^+},
\qquad 2^k r \gtrsim 1
\end{equation}
which, after summation with respect to $k$ gives the bound
\[
 | \bu^l_{hi}| \lesssim \sum_{2^k \ges r^{-1}} |\bu_k^l(r,t)|  \les 
\left(\frac{r}{\la r+ t \ra}\right)^{N}
\frac{1}{\la r+ t\ra ^{\frac12} \la r-t \ra^{\frac32} \log\la r-t \ra}.
\]
The bounds for the time derivative are obtained from the explicit formula
\[
\partial_t \bu^l = \int_0^\infty \phi_\xi(r)
(-f_0(\xi) \sin(t\xi) + \frac{1}{\xi} f_1(\xi) \cos(t\xi)) d\xi
\]
which shows that we produce an extra $2^k$ factor in \eqref{bulhia}. Similarly, an $r$ 
derivative applied to $\phi_\xi$ yields an additional $2^k$ factor in the asymptotic expansion.
Thus we obtain
\begin{equation}\label{bulhib}
|\partial_t \bu^{l}_k|+ |\partial_r \bu^{l}_k|\les  \frac{r ^{-\frac12} 2^{\frac{5k}2}}{\la 2^k |r-t| \ra^N   \la k^{-} \ra } 2^{-Nk^+},
\qquad 2^k r \gtrsim 1
\end{equation}
which leads to 
\[
 |\partial_t \bu^l_{hi}|+|\partial_r \bu^l_{hi}|  \les 
\left(\frac{r}{\la r+ t \ra}\right)^{N}
\frac{1}{\la r+ t\ra ^{\frac12} \la r-t \ra^{\frac52} \log\la r-t \ra}.
\]

We now consider  the terms in $\bu^{l,j}_{low}$. The main
contribution comes from $f_1$, so we take $f_0=0$ for convenience.
For $j=0$ we have
\[
\bu^{l,0}_{low}= \phi_0(r) \sum_k 
\chi_{\les 2^{-k}}(r) \int \xi^{-2} q(\xi) \sin(t\xi)  \hat{f}_1(\xi) \chi_k(\xi) d\xi
:=\phi_0(r) \sum_k 
\chi_{\les 2^{-k}}(r) g_k^0(t) :=\phi_0(r) g^0(r,t)
\]
Using stationary phase and the properties of $q$ we have
\[
|g^0_k(t)| + 2^{-k} |\partial_tg^0_k(t)| \les \frac{2^k}{\la k^- \ra^2 \la 2^k t \ra^N} 2^{-Nk^+}
\]
By summing with respect to $k$ we obtain 
\begin{equation}\label{bu0low}
 |g^0(r,t)| + \la t+r\ra (|\partial_rg^0(r,t)| + |\partial_tg^0(r,t)| )  \les  \frac{1}{\la t+r\ra
\log^2\la t+r\ra}.
\end{equation}
which yields the $\bu^{l,r}$ bound in \eqref{bulr}.

For $j \geq 1$ we have
\[
\begin{split}
u^{l,j}_{low} = & \ \sum_k 
\chi_{\{r \les 2^{-k}\}}
\frac1r \int \xi^{-2} q(\xi) \sin(t\xi) \sum_{j \geq 1} (r\xi)^{2j} \phi_j(r^2)  \hat{f}_1(\xi) \chi_k(\xi) d\xi
\\:= & \  r^{2j-1} \phi_j(r^2) \sum_k 
\chi_{\les 2^{-k}}(r)  g_k^j(t) :=   r^{2j-1} \phi_j(r^2) g^j(r,t)
\end{split}
\]
By stationary phase and the properties of $q$ and $\hat f_1$  we have
\[
|g^j_k(r,t)| 
+ 2^{-k} (|\partial_t g^j_k(r,t)|+|\partial_r g^j_k(r,t)|)  \les \frac{2^{(2j+1)k}}{\la k^- \ra^2 \la 2^k t \ra^N} 2^{-Nk^+}
\]
Summing up over $k$ we obtain
\begin{equation}
|g^j(r,t)| + \la t+r\ra (|\partial_rg^j(r,t)| + |\partial_tg^j(r,t)| ) \les  \frac{1}{\la t+r\ra^{2j+1} \log^2\la t+r\ra}
\end{equation}
Hence, using the bound \eqref{derphi} for $\phi_j$ we obtain
a bound for  $\bu^{l,j}_{low}$, namely
\begin{equation}\label{bujlow}
 |\bu^{l,j}_{low}(r,t)|+ |r \partial_r \bu^{l,j}_{low}(r,t)| + \la t+r \ra |\partial_t \bu^{l,j}_{low}(r,t)|
\lesssim \frac{C^j}{j!}  \frac{r^{2j-1} \log(1+r^2)}{\la t+r\ra^{2j+1} \log^2\la t+r\ra}
\end{equation}
Thus these contributions satisfy the bounds required of  $\bu^{l,nr}$.

We now turn our attention to the estimate \eqref{ulrt}, which applies in the region 
where $ r \approx t$. By \eqref{bulr} (for $\bu^l$) and by \eqref{bujlow}, 
the contributions of the term $\bu_{low}^l$ are all below the required threshold, 
so it remains to consider $\bu_{hi}^l$. We have 
\[
\begin{split}
\bu^l_{hi}(r,t) & = \int_0^\infty  \chi_{\ges r^{-1}}(\xi)
 \xi^{-1} \phi_\xi(r) (f_0(\xi) \cos(t\xi) + \frac{1}{\xi} f_1(\xi) \sin(t\xi)) d\xi 
\end{split}
\]
For $\phi_\xi$  we use the representation \eqref{phipsi} with $\phi_\xi^+$ 
as in \eqref{repphi+},
\[
\phi_\xi  = r^{-\frac12}(a(\xi) \sigma(r\xi,r) e^{ir\xi} +  \bar a(\xi)  \bar \sigma(r\xi,r) e^{-ir\xi}),
\qquad r\xi \gtrsim 1
\]
We notice that the operator $\partial_r + \partial_t$ kills the resonant factors 
$e^{\pm i(r - t)\xi}$ factors.  Precisely, we have 
\[
(\partial_r + \partial_t  +\frac{1}{2r} ) \phi_\xi(r)  \sin(t\xi) 
= 2r^{-\frac12}  \Re \left( e^{i\xi(r+t)} \xi a(\xi) \sigma(r\xi,r)\right)
 -  r^{-\frac12} \Re   \left( e^{ir\xi} a(\xi) \partial_r \sigma(r\xi,r) \right) \sin(t\xi)
\]
and a similar computation where $\sin(t\xi)$ is replaced by $\cos(t\xi)$.
This leads to 
\[
\begin{split}
 (\partial_r + \partial_t+\frac{1}{2r}) \bu^l_{hi}= &
    \int_0^\infty \chi_{\ges r^{-1}}(\xi) 
r^{-\frac12} \Re \left( 2i \xi e^{i(r+t)\xi} a(\xi) \sigma(r\xi,r)+ e^{ir\xi} a(\xi) \partial_r \sigma(r\xi,r) \cos(t\xi) \right) \frac{f_0(\xi)}{\xi} d\xi \\
 + & \int_0^\infty  \chi_{\ges r^{-1}}(\xi)  r^{-\frac12}
 \Re \left( 2i \xi e^{i(r+t)\xi} a(\xi) \sigma(r\xi,r) +  e^{ir\xi} a(\xi) \partial_r \sigma(r\xi,r) \sin(t\xi) \right) \frac{f_1(\xi)}{\xi^2} d\xi \\
\end{split}
\]
The two integrals above are treated as before, using stationary phase. 
The first term in each of the last integrals has a nonresonant phase, therefore 
each integrating by parts gains a factor of $(\xi t)^{-1}$.  Thus, taking 
\eqref{f01} into account, their contributions 
can be estimated by 
\[
\int_0^\infty \chi_{\ges t^{-1}}(\xi)  t^{-\frac12}  \xi (t\xi)^{-N}   \frac{\xi^{\frac52}}{\xi^2 \log \xi} d\xi 
\approx  \frac{1}{t^3 \log t} 
\]
The  second term contains the expression
$\partial_r \sigma(r\xi,r)$ which (see the description of 
$\sigma$ in Section~\ref{s:ft}) brings an additional factor of 
$r^{-1}(r\xi)^{-1} \approx t^{-2} \xi^{-1}$.  The contribution
of the part with phase $e^{i\xi(r+t)}$ is better than above, while the contribution 
of the part with phase $e^{i \xi(r-t)}$ is of the form
\[
\int_0^\infty  \chi_{\ges t^{-1}}(\xi)  a(\xi)
t^{-\frac12}  t^{-1} (t\xi)^{-1}   e^{i\xi(t-r)} \frac{\xi^{\frac52}}{\xi^2 \log \xi} d\xi 
\approx  \frac{1}{t^\frac52 \la t-r \ra^\frac12 \log \la t-r\ra} 
\]
as desired.

b)  We find $u^{nl}$  from the equation \eqref{bunl}  using a fixed point argument 
in the Banach space $Z^{nl}$  with norm
\[
\| f\|_{Z^{nl}} = \|h_1^{-1} t^{1.5} f \|_{L^\infty}
\]
Denoting by $Z^l$ the Banach space of functions of the form $\bu^{l,r} + \bu^{l,nr}$
 with norm as in \eqref{bubound}-\eqref{bulr}, we will show that the map
\[
T: (u,v) \to  L^{-1} N(u,v) = h_1(r) \int_r^\infty \frac{N(u, v)}{h_1(s)} ds
\]
is locally Lipschitz from $Z^l \times Z^{nl}$ into $Z^{nl}$, with a Lipschitz 
constant which can be made small if either both arguments are small or
$v$ is small  and the time $t$ is large enough, depending on the size of $u$.
This would imply the existence and uniqueness of $\bu^{nl}$, as well as 
its Lipschitz dependence on $\bu^l$ and implicitly on $(w_0,w_1)$.
Recall that 
\[
N(u,v) = \frac1r \big[  \sin{Q} \cdot (\cos (u + v)-1) + \cos{Q} \cdot (\sin (u + v) - (u + v)) \big]
\]
Then
\[
\begin{split}
|N(u,v)| & \les \frac1{r^2+1} ( |u|^2 + |v|^2) +  \frac1r (|u|^3 + |v|^3) \\
|\nabla N(u,v)| & \les \frac1{r^2+1} ( |u| + |v|) +  \frac1r (|u|^2 + |v|^2) 
\end{split}
\]
Hence it remains to show that 
\[
\int_{0}^\infty \frac1{r} ( |u|^2 + |v|^2) +  \frac{r^2+1}{r^2} (|u|^3 + |v|^3) dr \lesssim 
t^{-1.5} (\| u\|_{Z^l}^2 + \| v \|_{Z^{nl}}^2 + \| u\|_{Z^l}^3 + \| v \|_{Z^{nl}}^3)
\]
For $u$ we have two components $u^r$ and $u^{nr}$, therefore we need 
to consider the following six integrals:
\[
 \begin{split}
\int_0^\infty  \frac1{r}  |u^r|^2  dr \lesssim &
 \int_0^\infty \frac1{r} \frac{h_1^2(r)}{(t \log^2 t)^2} dr  
 \cdot \| u\|_{Z^l}^2  \approx \frac{1}{t^2 \log^4 t}   \| u\|_{Z^l}^2
\\
\int_0^\infty  \frac1{r}  |u^{nr}|^2  dr \lesssim &
 \int_0^\infty \frac1{r} \frac{r^2}{(t+r)^2 t \la t-r \ra^3  \log^2\la  t -r \ra} 
dr  \cdot \| u\|_{Z^l}^2  \approx \frac{1}{t^2}   \| u\|_{Z^l}^2
\\
\int_0^\infty  \frac1{r}  |v|^2  dr \lesssim &
 \int_0^\infty \frac1{r} h_1^2(r)  t^{-3} 
dr \cdot \| v\|_{Z^{nl}}^2  \approx \frac{1}{t^3}   \| v \|_{Z^{nl}}^2
\\
\int_0^\infty  \frac{r^2+1}{r^2}  |u^r|^3  dr \lesssim &
 \int_0^\infty  \frac{r^2+1}{r^2}  \frac{h_1^3(r)}{(t \log^2 t)^3} 
dr \cdot \| u\|_{Z^l}^3  \approx \frac{1}{t^3 \log^6 t}   \| u\|_{Z^l}^3
\\
\int_0^\infty  \frac{r^2+1}{r^2}   |u^{nr}|^3  dr \lesssim &
 \int_0^\infty  \frac{r^2+1}{r^2}\frac{r^3}{(t+r)^3 t^{\frac32} \la t-r \ra^\frac92  \log^3\la  t -r \ra} 
dr  \cdot \| u\|_{Z^l}^3  \approx \frac{1}{t^{1.5}}   \| u\|_{Z^l}^3
\\
\int_0^\infty   \frac{r^2+1}{r^2}|v|^3  dr \lesssim &
 \int_0^\infty  \frac{r^2+1}{r^2}  h_1^3(r)  t^{-4.5} 
dr  \cdot \| v\|_{Z^{nl}}^3  \approx \frac{1}{t^{4.5}}   \| v \|_{Z^{nl}}^3
\end{split}
\]
We remark that the worst decay $t^{-1.5}$ comes from the fifth integral above;
all other terms are better.

The argument for $\partial_t \bu^{nl}$ is more involved.  Differentiating the equation 
\eqref{nleq} we obtain
\begin{equation}\label{longdtunl}
\begin{split}
L (\partial_t \bu^{nl}+\frac{h_1}6(\bu^l)^3) = &\  N_u( \bu^{l},\bu^{nl}) \partial_t \bu^{l} +  N_v
( \bu^{l},\bu^{nl}) \partial_t \bu^{nl} + \frac{h_1}6 \partial_r (\bu^l)^3
\\
= & \ N_v( \bu^{l},\bu^{nl}) (\partial_t \bu^{nl}+\frac{h_1}6(\bu^l)^3) +
 [N_u( \bu^{l},\bu^{nl}) - \frac{ h_1}2 (\bu^l)^2]  \partial_t \bu^{l}  \\ &\
-\frac16 N_v( \bu^{l},\bu^{nl}) h_1(\bu^l)^3 + 
\frac{h_1}{6} (\partial_t + \partial_r) (\bu^l)^3
\end{split}
\end{equation}
The approach is similar to what we have done before. We adjust the base space to 
\[
\| f\|_{\tilde Z^{nl}} = \|h_1^{-1} t^{2} f \|_{L^\infty}
\]
and continue with the same steps. By the previous computation the first term on the right is perturbative.
The main cancellation occurs in the second term, where the $(\bu^l)^2$ term 
disappears. Precisely, we have
\[
N_u(u,v) - \frac12 
h_1 u^2 =  \frac{2}{1+r^2} \sin(u+v) - \frac{1-r^2}{r(1+r^2)} (1 -\cos(u+v)) - 
\frac{r}{1+r^2} u^2  
\]
therefore
\[
|N_u(u,v) - \frac12 h_1 u^2| \lesssim \frac{1}{1+r^2}(|u|+|v|) + \frac{1}r (|u|^3+ |u||v|+|v|^2) 
+ \frac{1}{r(1+r^2)} |u|^2 
\]
For $  \partial_t \bu^{l}$ we use the same bounds as for $\bu^l$.
Then, compared with  the previous computation, we need to reestimate the terms 
involving $|u|^3$, $|u||v|$ and $|u|^2$.  The resonant part of $u$ yields 
better bounds, so we only estimate terms involving $u^{nr}$:
\[
 \begin{split}
\int_0^\infty  \frac{r^2+1}{r^2}  |u^{nr}|^4  dr \lesssim &\   \| u\|_{Z^l}^4 \cdot 
 \int_0^\infty \frac{r^2+1}{r^2} \frac{r^4}{(t+r)^4 t^2 \la t-r \ra^{6}  \log^4\la  t -r \ra} 
dr   \approx \frac{1}{t^2}   \| u\|_{Z^l}^4
\\
\int_0^\infty  \frac{r^2+1}{r^2}  |u^{nr}|^2 |v| \  dr \lesssim &\   \| u\|_{Z^l}^2 \|v\|_{Z^{nl}} \cdot 
 \int_0^\infty  \frac{r^2+1}{r^2} \frac{r^2}{(t+r)^2 t^{2.5} \la t-r \ra^{3}  \log^2\la  t -r \ra} 
dr   \approx \frac{1}{t^{2.5}}    \| u\|_{Z^l}^2 \|v\|_{Z^{nl}}
\\
\int_0^\infty  \frac1{r^2}  |u^{nr}|^3  dr \lesssim &\   \| u\|_{Z^l}^3 \cdot 
 \int_0^\infty \frac{1}{r^2} \frac{r^3}{(t+r)^3 t^{1.5} \la t-r \ra^{4.5}  \log^3\la  t -r \ra} 
dr   \approx \frac{1}{t^{3.5}}   \| u\|_{Z^l}^4
\end{split}
\]
The third term on the right in \eqref{longdtunl} is better behaved than the second.
Finally, for the last term in \eqref{longdtunl} we invoke \eqref{ulrt} so that we use the same bounds for 
$(\partial_t + \partial_r) (\bu^l)$ as for $r^{-1} \bu^l$. Then the integral to estimate
is 
\[
\int_0^\infty  \frac1{r}  |u|^3  dr \lesssim \frac{1}{t^{2.5}}   \| u\|_{Z^l}^3
\]

c) In the case of $\bu^l$ this part follows from the linearity. In the case 
of $\bu^{nl}$ the Lipschitz dependence on $\bu^l$ has already been discussed above.
An additional argument is required for $\delta \partial_t \bu^{nl}$. However, nothing new happens there, and the details are left for the reader.

\end{proof}

\section{The transition between $\tw$ and $\tu$}
\label{s:transition}

In this section study the transition from $\tw$ to $\tu$, which were
both introduced in \eqref{egdef}. This transition is described by
\eqref{ge}, which we recall for convenience
\[
\tw = \partial_r \tu  - \frac{\sin(\tu + \bu) - \sin \bu}r
\]
The main result of this section is the following

\begin{proposition} \label{propeg} a) Assume that $\tw \in LX$ is small
  and $\bu,\bw$ are as in Proposition \ref{lul}.  Then
  for $t$ large enough there exists a unique solution $\tu \in X$ of
  \eqref{ge} which satisfies
\begin{equation} \label{geXLX}
\| \tu \|_X \lesss \| \tw \|_{LX} 
\end{equation}
Furthermore, $\tu$ has a Lipschitz dependence on both $\gamma$ and 
on the linear data $(w_0,w_1)$ for $\bw$,  
\begin{equation} \label{dgeXLX}
\| \delta \tu \|_X \lesss \| \delta \tw \|_{LX} + \frac{1}{t \log^2 t} 
\| (\delta w_0,\delta w_1) \|_S \| \tw \|_{LX},
\end{equation} 

Also, if $\tw$ is a function of $t$ then
 \begin{equation} \label{geXLXa}
 \| \partial_t \tu \|_X \lesss \| \partial_t \tw \|_{LX}+ \frac{1}{t \log^2 t} 
 \| \tw \|_{LX}
\end{equation}

b) Assume in addition that $\tw \in L^\infty$. Then
\begin{equation}
|\tu(r)| \lesss r \log r  \|\tw \|_{LX \cap L^\infty}, \qquad r \ll 1 
\label{tulowr}\end{equation}

\end{proposition}

\begin{proof} a) The equation \eqref{ge} is rewritten as
\begin{equation} \label{twef}
L \tu  = \tw+  \frac{ \sin(\tu + \bu) - \sin \bu - \cos Q \cdot \tu}{r} :=
\tw +  F(\tu, \bu-Q)
\end{equation}
Hence in order to prove both \eqref{geXLX} and \eqref{dgeXLX} it suffices 
to show that at fixed large enough time the map $F$ is Lipschitz 
\[
F: X \times (Z^l+Z^{nl}) \to LX
\]
with a small Lipschitz constant in the second variable. 
For the $X$ norm we use the embeddings \eqref{Xembt}-\eqref{linX4}. 
For the $LX$ norm we use \eqref{LXemb}, which shows that is enough to estimate 
$F(\bu,\tu)$ in $L^1 \cap L^2$.  We expand $F$ as follows:
\[
\begin{split}
F(\beta,v) & = \frac{ \sin(\beta + Q+ v) - \sin (Q+v) - \cos Q \cdot \beta}{r} \\
& = \frac{ (\cos (Q+v) - \cos Q) \cdot \beta}{r}
- \frac{ \sin (Q+v) \cdot \beta^2}{2r} + \frac{O(\beta^3)}{r} \\
& = \frac{ \sin Q \cdot v \beta}{r}
+ \frac{ \sin Q \cdot \beta^2}{2r} 
+ \frac{O (v^2 \beta)}{r}
+ \frac{O(\beta^3)}{r}  \\
\end{split}
\]
Hence
\begin{equation}\label{Feu}
|F(\beta,v)| \lesssim  \frac{|v| |\beta|}{1+r^2}
+ \frac{|\beta|^2}{1+r^2} + \frac{|\beta|^3}{r} + \frac{|v|^2 |\beta|}{r} \\
\end{equation}
By using \eqref{linX}, \eqref{Xembt} and \eqref{bulr}, we bound this first in $L^2$
\[
\begin{split}
\| F(\beta,v)\|_{L^2} \lesssim & \  \|\frac{\beta}{\log(1+ r)}\|_{L^2}
\left( \| \beta \|_{L^\infty} + \| \beta \|_{L^\infty}^2 + \|\frac{v}{1+r}\|_{L^\infty} 
+\|\frac{v^2\log(1+ r)}{r} \|_{L^\infty} \right)
\\ \lesssim & \ \| \beta \|_{X}^2+ \| \beta \|_{X}^3 + \| \beta \|_{X}( \frac{1}{t \log^2 t}
\| v \|_{Z^l+Z^{nl}} + \frac{\log t}{t^2}\| v \|_{Z^l+Z^{nl}}^2)
\end{split}
\]
and then in $L^1$,
\[
\begin{split}
\| F(\beta,v)\|_{L^1} \lesssim & \  \|\frac{\beta}{\log(1+ r)}\|_{L^2}^2
(1+\|\beta\|_{L^\infty})\\ &  +  \|\frac{\beta}{\log(1+ r)}\|_{L^2} 
(\|\frac{ v \log (1+r)}{1+r^2}\|_{L^2} + \| \frac{v^2\log(1+ r)}{r} \|_{L^2})
\\ \lesssim & \ \| \beta \|_{X}^2+  \| \beta \|_{X}^3 + \| \beta \|_{X}
(  \frac{1}{t \log^2 t}
\|v \|_{Z^l+Z^{nl}} + \frac{\log t}{t^\frac32}\|v\|_{Z^l+Z^{nl}}^2)
\end{split}
\]
Hence we obtain 
\[
\| F(\beta,v)\|_{LX} \lesssim \| \beta \|_{X}^2+  \| \beta \|_{X}^3 + \| \beta \|_{X}
(  \frac{1}{t \log^2 t}
\|v \|_{Z^l+Z^{nl}} + \frac{\log t}{t^\frac32}\|v\|_{Z^l+Z^{nl}}^2)
\]
A similar analysis yields
\[
\|\beta_1 F_{\beta}(\beta,v)\|_{LX} \lesssim \| \beta_1\|_{X}(\| \beta\|_{X}+  \| \beta \|_{X}^2 +
 \frac{1}{t \log^2 t} \|v\|_{Z^l+Z^{nl}} + \frac{\log t}{t^\frac32}\|v\|_{Z^l+Z^{nl}}^2)
\]
respectively 
\[
\|v_1 F_{v}(\beta,v)\|_{LX} \lesssim \| v_1\|_{Z^l+Z^{nl}}  \| \beta \|_{X}(  
 \frac{1}{t \log^2 t}  + \frac{\log t}{t^\frac32}\|v\|_{Z^l+Z^{nl}})
\]
 By the contraction principle this proves both \eqref{geXLX} and \eqref{dgeXLX}.
The time decaying factors guarantee that for any size of $\bu-Q$ the problem can be solved 
for large enough time. 

To prove \eqref{geXLXa} we differentiate with respect to $t$ in \eqref{twef},
\[
L \partial_t \tu = \partial_t \tw + F_\tu(\tu,\bu) \partial_t \tu +
 F_{\bu}(\tu,\bu) \partial_t \bu
\]
Since $\partial_t \bu$ satisfies the same pointwise bounds as $\bu$,
the last two estimates above show that the contraction principle still
applies.

b) Due to the embedding $X \subset \dot H^1_e \subset L^\infty$ we already have 
a small uniform bound for $\tu$. We solve the ode \eqref{twef} in $[0,1]$ with 
Cauchy data at $r=1$. Making the bootstrap assumption 
\begin{equation}
|\tu| \leq M r |\log (r/2)|
\label{boot:tu}\end{equation}
we rewrite the equation \eqref{twef} 
in the form 
\[
|L\tu - \tw | \leq M^3 r^2 |\log^3(r/2)| + C,
\qquad C \approx_S \|\tu\|_{L^\infty}
\]
Then solving the linear $L$ evolution we have
\[
|\tu| \lesssim r (|\tw(1)| + M^3) +  C r |\log (r/2)|
\lesssim_S M^3 r +  r |\log (r/2)|\|\tu\|_{L^\infty}
\]
If $\|\tu\|_{L^\infty}$ is sufficiently small then we can 
choose $M$ small enough so that the above bound is stronger than
our bootstrap assumption \eqref{boot:tu}. The proof of \eqref{tulowr}
is concluded.
 
\end{proof}

\section{The nonlinearity in the $\gamma$ equation}
\label{s:nonl}

Our main goal is to solve the equation \eqref{geq} for $\gamma$ with
zero Cauchy data at $t=\infty$. For the linear $\tilde H$ wave
equation we use the $LX$ bounds in Lemma~\ref{l:Kest}. The auxiliary
function $\epsilon$ is uniquely determined by $\gamma$ via
Proposition~\ref{propeg}.  In this section we estimate the nonlinear
contribution in \eqref{geq}, namely
\[
N(w,u)=\frac{2(\cos Q- \cos u)}{r^2} w + \frac{1}{r} \sin u (u_t^2 - w^2)
\]
In light of the decompositions $u=\bu + \tu, w = \bw + \tw$, this nonlinearity has 
three types of contributions,
\[
N(w,u) = N(\bw,\bu) + N^l(\bw,\bu,\tw,\tu) + N^n(\bw,\bu,\tw,\tu)
\]
The first one, $N(\bw,\bu)$, should be seen as an inhomogeneous term.
The reason we need to consider this separately is that $\bu, \bw$ have
a different behavior compared to $\tu, \tw$ as $t$ goes to infinity.
The term $N^l$ contains the linear contributions in $\tu,\tw$ in the
difference $N(w,u) - N(\bw,\bu)$,
\[
N^l= \frac{2(\cos Q- \cos \bu)}{r^2} \tw+ \frac{2 \sin \bu \cdot \tu}{r^2} \bw 
+ \frac{\sin \bu (2 \bu_t \tu_t - 2 \bw \tw) + \cos \bu \cdot \tu(\bu_t^2 - \bw^2)}{r}
\]
The remaining term $N^l$ contains the genuinely nonlinear
contributions in $\tu,\tw$ in the difference $N(w,u) - N(\bw,\bu)$,
\[
\begin{split}
N^n= &\frac{2(\cos \bu - \cos u-\sin \bu \cdot \tu)}{r^2} \bw + \frac{2(\cos \bu - \cos(\bu + \tu))}{r^2} \tw 
+ \frac{\sin \bu (\tu_t^2 - \tw^2)}{r} \\
& + \frac{(\sin u - \sin \bu)(2 \bu_t \tu_t - 2 \bw \ga + \bu_t^2 - \bw^2)}r
+\frac{(\sin u - \sin \bu - \cos \bu \cdot \tu)(\bu_t^2 - \bw^2)}{r}
\end{split}
\]
Our main result is the following
\begin{proposition} \label{p:mn}

a) If $\bu,\bw$ are  as in Proposition~\ref{lul}  then for $t \gtrsim_S 1$ we have
\begin{equation} \label{Nbubw}
t^{1.5} \| K N(\bu,\bw) \|_{LX}  
+  t^{2.5} (\| \partial_t K N(\bu,\bw) \|_{LX} + \| K N(\bu,\bw) \|_{\dHe} ) \lesss  1
\end{equation}
with Lipschitz dependence on the initial data $(w_0,w_1)$ for $\bw$.

b) If we assume the following
\[
\sup_t t^{1.5} (\| \tw(t) \|_{LX}  + \| \tu(t) \|_{X})  
+ \sup_t t^{2.5} (\| \partial_t \tw(t) \|_{LX} + \| \tw \|_{\dHe} + \| \partial_t \tu(t) \|_{X}) \les M
\]
then
\begin{equation} \label{NLest}
t^{1.5}  \| K N^l \|_{LX}   +  t^{2.5} (\| \partial_t K N^l \|_{LX} + \| K N^l \|_{\dHe}) 
\lesss  M  (\log t)^{-2}
\end{equation}
\begin{equation} \label{NNest}
t^{1.5} \| K N^n \|_{LX}+  t^{2.5} (\| \partial_t K N^n \|_{LX} + \| K N^n \|_{\dHe}) 
\lesss  (M^2+M^3) t^{-.5} \log t
\end{equation}
In addition, the maps $(\tu,\tw) \to (KN^l,KN^n)$ satisfy similar Lipschitz bounds.
\end{proposition}

When combined with Proposition \ref{propeg} this result allows us to
treat the problem in $\tw,\tu$ perturbatively. The additional gains in
$t$ decay in \eqref{NLest} and \eqref{NNest} allows us to consider
large Schwartz perturbations of the soliton. We note that in the case 
of $KN^l$ we gain only logarithms. This implies that for large 
Schwartz data $(w_0,w_1)$ in the linear equation our solutions 
are only defined for $t > T$ with $T$ exponentially large. 

\subsection{The term $N(\bw,\bu)$.} Our goal here is to prove the
estimate
\begin{equation}\label{nbubw}
\| N(\bw,\bu)\|_{LX} \lesss t^{-3.5}
\end{equation}
Then the bound \eqref{Nbubw} is a consequence of 
Lemma~\ref{l:Kest}. To prove this we split
\[
N(\bw,\bu) = \chi_{r \ll t} N(\bw,\bu) + \chi_{r \gg t} N(\bw,\bu) + 
\chi_{r \approx t} N(\bw,\bu)=N_1 + N_2 + N_3
\]
For the first two terms it suffices to use a direct estimate
\[
|N(\bw,\bu)| \lesssim \frac{\sin Q}{r^2} |\bu-Q||\bw| + \frac{1}{r^2} |\bu-Q|^2|\bw|
+ \frac{1}{r} (\sin Q +|\bu|)(|\bu_t|^2 + |\bw|^2) 
\]
Using the bounds \eqref{bulr} and \eqref{unl} for $\bu -Q$, as well as 
the bound \eqref{bwpoint} for $\bw$, this gives
\[
|N_1(\bw,\bu)| \lesss \chi_{r \ll t}  \frac{1}{\la r \ra^4  t^{4}}
\]
where the leading contribution comes from $u^{l,r}$.
This implies that 
\[
\| N_1 \|_{L^1 \cap L^2}  \lesss t^{-4}
\]
which suffices for \eqref{nbubw} in view of the embedding \eqref{LXemb}.
Similarly
\[
|N_2| \lesss \chi_{r \gg t}  \frac{1}{\la r \ra^{8}}
\]
which also gives 
\[
\| N_2 \|_{L^1 \cap L^2}  \lesss t^{-4}
\]

However, a similar direct computation for $N_3$ 
only gives 
\[
|N_3(\bw,\bu)| \lesss \chi_{r \sim t}  \frac{1}{  t^{2.5} \la t-r \ra^{5.5}}
\]
which fails by two units, 
\[
\| N_3 \|_{L^1 \cap L^2}  \lesss t^{-1.5}
\]
Hence in order to conclude the proof of \eqref{nbubw} we need to better exploit 
the structure of $N$ and capture a double cancellation on the null cone. In the computations below (through the
end of the subsection) we work in the regime $r \approx t$. We expand $N(\bw,\bu)$ as 
\[
\begin{split}
N(\bw,\bu) = & \  2 \frac{\sin Q}{r^2} (\bu-Q) \bw +  \frac{\cos Q}{r^2} (\bu-Q)^2 \bw
+ \frac{\sin Q}{r} (\bu_t^2-\bw^2) + \frac{\cos Q}{r} (\bu_t^2-\bw^2)(\bu -Q)
\\ & \ 
+ \frac{\sin Q}{r^2} w O((\bu-Q)^3) + \frac{\cos Q}{r^2} (w O((\bu-Q)^4)
\\ & \
+ \frac{\sin Q}{r}(\bu_t^2-\bw^2) O((\bu-Q)^2) + \frac{\cos Q}{r}(\bu_t^2-\bw^2) O((\bu-Q)^3)
\end{split}
\]
The terms on the second line are already acceptable, i.e. can be estimated by 
$ t^{-4.5} \la t-r \ra^{-3.5}$. For further progress we observe that 
by \eqref{unl} we have 
\[
\bu^{nl} = O_S( \la t \ra)^{-2.5}, \qquad   \partial_t \bu^{nl} = O_S(  t^{-2.5} \la t-r \ra^{-0.5})
\]
and that by \eqref{ulrt}   we can write
\begin{equation} \label{bcan}
\partial_t \bu  + \bw = \partial_t \bu^{nl} + \partial_t \bu^l + \partial_r \bu^l + \frac{\cos Q}r 
\bu^l = O_S( t^{-1.5} \la t-r \ra^{-1.5})
 \end{equation}
 The first relation above allows us to dispense with $\bu^{nl}$
 everywhere and replace $\bu-Q$ by $\bu^l$, and the second allows us
 to estimate the third line in $N(\bw,\bu)$. We are left with 
\[
N(\bw,\bu) =  2 \frac{\sin Q}{r^2} \bu^l \bw +  \frac{\cos Q}{r^2} (\bu^l)^2 \bw
+ \frac{\sin Q}{r} ((\bu^l_t)^2-\bw^2) + \frac{\cos Q}{r} ((\bu_t^l)^2-\bw^2) \bu^l
+ O_S(  t^{-4.5} \la t-r \ra^{-3.5}) 
\]
To advance further we substitute $\bw = \partial_r \bu^l - \frac{\cos Q}{r} \bu^l$
everywhere. The $\frac{\cos Q}{r} \bu^l$ is acceptable in the first two terms 
of $N$, i.e. it gives contributions of $O_S(  t^{-4.5} \la t-r \ra^{-3.5})$,
and we discard it. For the last two terms we use the better approximation from \eqref{ulrt}
\[
\bu_t^l = - \partial_r \bu^l - \frac{1}{2r} \bu^l + O( t^{-2.5} \la t-r \ra^{-0.5})
 \]
Then we can write
\[
\begin{split}
(\bu_t^l)^2-\bw^2 = & \ (\partial_r \bu^l + \frac{1}{2r} \bu^l)^2 - 
(\partial_r \bu^l - \frac{\cos Q}{r} \bu^l)^2
+ O_S( t^{-3} \la t-r \ra^{-3})
\\
= & \ - \frac{1}{r} \bu^l \partial_r \bu^l  + O_S(t^{-3} \la t-r \ra^{-3})
\end{split}
\]
It is also harmless to replace $\sin Q$ by $r^{-1}$ and $\cos Q$ by $-1$
everywhere. Returning to $N$ we obtain
\[
\begin{split}
N(\bw,\bu) = & \  \frac{2}{r^3} \bu^l \partial_r \bu^l -  \frac{1}{r^2} (\bu^l)^2 \partial_r \bu^l
- \frac{1}{r^3} u^l \partial_r u^l + \frac{1}{r^2}  (\bu^l)^2 \partial_r \bu^l
+ O_S( \la t \ra^{-4.5} \la t-r \ra^{3.5}) 
\\ = & \  \frac{1}{2r^3} \partial_r (\bu^l)^2 + O_S(  t^{-4.5} \la t-r \ra^{-3.5}) 
\end{split} 
\]
in the region $r \approx t$, which we rewrite as
\[
N_3 = Lg +  \chi_{r \approx t} O_S( t^{-4.5} \la t-r \ra^{-3.5}) , \qquad 
g =  \chi_{r \approx t}\frac{1}{2r^3} (\bu^l)^2 
\]
The last term can be directly estimated in $L^1 \cap L^2$. For the leading term 
$Lg$ we estimate $g$ in $H^1_e$ and use the embedding \eqref{Xembt}.
We have 
\[
|g| \lesss \frac{1} { t^{4} \la t-r \ra^{3}},  \qquad |\partial_r g| \lesss
 \frac{1} { t^{4} \la t-r \ra^{4}}
\]
therefore
\[
\| g\|_{H^1_e} \lesss   \frac{1}{ t^{3.5}} 
\]
This concludes the proof of \eqref{nbubw}.

\subsection{The bound for $N^l$.}
Our goal here is to establish the bound
\begin{equation} \label{nlinlx}
\| N^l\|_{LX} \lesss \frac{1}{t^{3.5} \log^2 t} M
\end{equation}
We recall that
\[
N^l= \frac{2(\cos Q- \cos \bu)}{r^2} \gamma+ \frac{2 \sin \bu \cdot \epsilon}{r^2} \bw 
+ \frac{\sin \bu (2 \bu_t \epsilon_t - 2 \bw \ga) + \cos \bu \cdot \epsilon(\bu_t^2 - \bw^2)}{r}
\]
The pointwise estimate
\[
|\frac{2(\cos Q- \cos \bu)}{r}| \les \frac{1}{r^2+1} |\bu -Q| + \frac1{r} |\bu-Q|^2 
\]
combined with the pointwise bounds for $\bu$ from \eqref{bulr} leads to
\[
\|\frac{2(\cos Q- \cos \bu)}{r}\|_{L^\infty \cap L^2} \lesss \frac{1}{t \log^2 t}
\]
with the worst contribution arising from the resonant part of $\bu$.
From \eqref{LXemb} it follows that
\[
\|\frac{2(\cos Q- \cos \bu)}{r^2} \ga \|_{LX} \les \|\frac{2(\cos Q- \cos \bu)}{r}\|_{L^\infty \cap L^2} 
\cdot \| \frac{\ga}{r} \|_{L^2} \lesss \frac{1}{t^{3.5} \log^2 t} M.
\]
Next, from \eqref{bulr} and \eqref{bwpoint}, it follows that
\[
\| \frac{\bu \cdot \bw}{r^2} \log(2+r) \|_{L^\infty \cap L^2} \lesss \frac{\log  t }{ t ^{2.5}}
\]
which combined with (recall \eqref{linX})
\[
\| \frac{\epsilon}{\log (2+r)} \|_{L^2} \les \| \epsilon \|_{X} \les  t^{-1.5} M
\]
gives
\[
\|\frac{2 \sin \bu \cdot \epsilon}{r^2} \bw \|_{LX} \lesss  \frac{\log t}{t^{4}} M
\]
Using \eqref{bulr} we obtain
\[
\| \frac{\bu \bu_t}{r} \log(2+r) \|_{L^\infty \cap L^2} \lesss  \frac{\log t} {t^{1.5}} 
\]
therefore by invoking \eqref{LXemb} and \eqref{linX}, it follows that 
\[
\| \frac{\sin(\bu) \cdot \bu_t \epsilon_t}{r} \|_{LX}  
\les \| \frac{\bu \bu_t}{r} \log(2+r) \|_{L^\infty \cap L^2} \| \frac{\epsilon_t}{\log(2+r)} \|_{L^2} 
\lesss \frac{\log t}{t^{4}} M
\]

The following term requires some extra work. Using \eqref{bulr} and \eqref{bwpoint}, we note that
away from the cone we have $|\sin (\bu)| \lesssim \sin Q$ and continue with
\[
\| \chi_{r \not \approx t} \frac{\bw \sin \bu}{r}\|_{L^1 \cap L^2}
\lesss t^{-2}
\]
followed by
\[
\|\chi_{r \not \approx t} \frac{\sin(\bu) \cdot \bw \ga}{r} \|_{LX} \les 
\| \chi_{r \not \approx t} \frac{\bu \bw}{r} \|_{L^1 \cap L^2} \| \ga \|_{L^\infty}  \lesss t^{-4.5} M
\]
Near the cone we write
\[
\begin{split}
\chi_{r  \approx t}  \frac{\bw \sin \bu}{r} = &\chi_{r  \approx t} \left( \ \frac{2\bw }{1+r^2} - \frac{\bw (\bu-Q)}{r} \cos{Q} 
+   \frac{\bw O((\bu-Q)^2)}{1+r^2} +  \frac{\bw O((\bu-Q)^3)}{r})\right)
\\
= & \ \chi_{r  \approx t} \frac{\bw (\bu-Q)}{r} + O_S( t^{-2.5} \la t-r\ra^{-2.5}) 
\\= & \  L (\chi_{r  \approx t}  r^{-1} (\bu^l)^2)  + O_S( t^{-2.5} \la t-r\ra^{-2.5}) 
\end{split}
\]
The second term is estimated as above in $L^1 \cap L^2$ and yields 
a contribution of $t^{-4} M$ to the $\|N^l\|_{LX}$ bound.
For the first term we write its contribution to $N^l$ in the form
\[
L (\chi_{r  \approx t}  r^{-1} (\bu^l)^2) \tw =  L(\chi_{r  \approx t}  r^{-1} (\bu^l)^2 \tw) +  
\chi_{r  \approx t} r^{-1} (\bu^l)^2 \partial_r \tw
\]
Then, using \eqref{Xembt} for the first term and \eqref{LXemb} for the second term,
we have
\[
\begin{split}
\| L (\chi_{r  \approx t}  r^{-1} (\bu^l)^2) \tw \|_{LX}
\lesssim & \ \| \chi_{r  \approx t}  r^{-1} (\bu^l)^2 \tw\|_{H^1_e} + \| \chi_{r  \approx t} r^{-1} (\bu^l)^2 \partial_r \tw\|_{L^1 \cap L^2}
\\
\lesssim & \ \| \chi_{r  \approx t}  r^{-1} (\bu^l)^2\|_{H^1_e} \| \tw\|_{\dot H^1_e} + 
\| \chi_{r  \approx t} r^{-1} (\bu^l)^2\|_{L^2 \cap L^\infty} \| \partial_r \tw\|_{L^2}
\\
\lesss & \ t^{-1.5} \| \tw\|_{\dot H^1_e} \lesss t^{-4} M
\end{split}
\]
It remains to bound the last term in $N^l$. For this we 
take advantage of the first order cancellation on the
cone in the expression $\bu_t - \bw$, see \eqref{bcan}, 
which combined with \eqref{bulr} and \eqref{bwpoint}, gives
\[
\|  \frac{\cos \bu(\bu_t^2 - \bw^2) \log (2+ r)}{r} \|_{L^2 \cap L^\infty} \lesss
\frac{\log t}{t^{2.5}}. 
\]
This leads to 
\[
\| \tu  \frac{\cos \bu(\bu_t^2 - \bw^2) }{r} \|_{L^1 \cap L^2}
\lesss \frac{\log t}{t^{2.5}} \| \frac{\tu}{ \log (2+r)}\|_{L^2} 
\lesss  \frac{\log t}{t^{2.5}} \| {\tu}\|_{X}  \lesss  \frac{\log t}{t^{4}}
M
\]
This concludes the proof of the $N^l$ bound \eqref{nlinlx}.

\subsection{The bound for $N^n$.} 
Our goal here will be to prove the bound
\begin{equation} \label{nn}
\|N^n\|_{LX} \lesss \frac{\log t}{t^4} (M^2+M^3)
\end{equation}
We recall
the expression of $N^n$:
\[
\begin{split}
N^n= &\frac{2(\cos \bu - \cos u-\sin \bu \cdot \tu)}{r^2} \bw + \frac{2(\cos \bu - \cos(\bu + \tu))}{r^2} \tw
+ \frac{\sin u (\tu_t^2 - \tw^2)}{r} \\
& + \frac{(\sin u - \sin \bu)(2 \bu_t \tu_t - 2 \bw \tw)}r
+\frac{(\sin u - \sin \bu - \cos \bu \cdot \tu)(\bu_t^2 - \bw^2)}{r}
\end{split}
\]

We successively consider the terms on the right. For the first one we start with
\[
| \frac{2(\cos \bu - \cos u-\sin \bu \cdot \tu)}{r^2} \bw | \les \frac{\tu^2 |\bw|}{r^2}.
\]
Then, using \eqref{bwpoint} and \eqref{linX}, we obtain
\[
\| \frac{\tu^2 \bw}{r^2} \|_{L^1 \cap L^2} \les 
\| \frac{\tu}{\log (2+r)} \|_{L^\infty \cap L^2} \| \frac{\tu}{\log(2+r)} \|_{L^2} 
\| \frac{\bw}{r^2} \log^2(2+r) \log \|_{L^\infty } 
\lesss  \frac{ \log^2 t}{t^{5.5}}  M^2
\]
The second term in $N^n$ is estimated by 
\[
|\frac{\cos \bu - \cos(\bu + \tu)}{r^2} \tw| \les \frac{|\sin \bu \cdot \tu \tw|}{r^2} 
+ \frac{ |\tu^2 \tw|}{r^2} 
\les \frac{|\tu \tw|}{r \la r \ra^2} + \frac{|(\bu-Q) \tu \tw|}{r^2} + \frac{ |\tu^2 \tw|}{r^2}.
\]
The first two terms can be estimated in $L^1 \cap L^2$ as before,
\[
\| \frac{\tu \tw}{r\la r\ra^2} \|_{L^1 \cap L^2} \les \| \frac{\tw}{r} \|_{L^2} \| \frac{\tu}{\la r \ra^2} \|_{L^\infty \cap L^2}
\lesss t^{-4} M^2
\]
\[
\|\frac{(\bu-Q) \tu \tw}{r^2}\|_{L^1 \cap L^2} \les \| \frac{\tw}{r} \|_{L^2} \| \frac{\bu-Q}{r} \|_{L^2 \cap L^\infty} \| \tu \|_{L^\infty}
\lesss t^{-5} M^2 
\]
For the last term we first get the $L^1$ bound
\[
\| \frac{\tu^2 \tw}{r^2} \|_{L^1} \les \| \tu \|_{L^\infty} \| \frac{\tu}{r} \|_{L^2} \| \frac{\tw}r \|_{L^2} 
\les \frac{ 1}{t^{5.5}} M^3
\]
However, getting the $L^2$ bound is more delicate:
\[
\| \frac{\tu^2 \tw}{r^2} \|_{L^2} \les  \| \frac{\tu}{\sqrt{r}} \|^2_{L^\infty} \| \frac{\tw}r \|_{L^2} 
\les \frac{ 1}{t^{5.5}} M^3
\]
where the pointwise bound for $\frac{\tu}{\sqrt{r}}$ near $r = 0$ comes from 
\eqref{tulowr}.

The third term in $N$ is estimated by using \eqref{bulr}
\[
| \frac{\sin u (\tu_t^2 - \tw^2)}{r}| \lesssim \frac{|\tu_t|^2}{1+r}+
 \frac{  |\tw^2|}{1+r}
\]
We successively consider all terms:
\[
\| \frac{|\tu_t|^2}{1+r}\|_{L^1 \cap L^2}  \lesssim 
\| \frac{\tu_t}{\log(2+r)}\|_{L^2 \cap L^\infty}
 \|\frac{\tu_t}{\log(2+r)}\|_{L^2} \lesssim \frac{1}{t^{5}}M^2
\]
\[
\| \frac{|\tw|^2}{1+r}\|_{L^1 \cap L^2} \lesssim  \| \tw\|_{L^2 \cap L^\infty}
\| \frac{\tw}{r}\|_{L^2} \lesssim \frac{1}{t^{4}}M^2
\]

Next we estimate the fourth term in $N^n$,
\[
|\frac{(\sin u - \sin \bu)(2 \bu_t \tu_t - 2 \bw \tw )}r| 
\les \frac{|\tu|( |\bu_t \tu_t| + | \bw \tw| )}r
\]
On behalf of \eqref{bwpoint}, \eqref{bulr} and \eqref{linX}, we have
\[
\| \frac{\tu \bu_t \tu_t}{r}  \|_{L^1 \cap L^2} \les \| \tu_t \|_{L^\infty} \| \frac{\tu}{\log(2+r)} \|_{L^2}  
\| \frac{\bu_t}{r} \log(2+r) \|_{L^\infty \cap L^2} \lesss \frac{\log t}{t^{4}} M^2
\]
\[
\| \frac{\tu \bw \tw}{r}  \|_{L^1 \cap L^2} \les \| \tu \|_{L^\infty} \| \bw \|_{L^2 \cap L^\infty}  
\| \frac{\tw}{r} \|_{L^2} \lesss t^{-4} M^2
\]

Finally we consider the last term in $N^n$,
\[
|\frac{(\sin u - \sin \bu - \cos \bu \cdot \tu)(\bu_t^2 - \bw^2)}{r}| \les \frac{\tu^2(\bu_t^2 + \bw^2)}{r}
\]
which, by using \eqref{bwpoint}, \eqref{bulr} and \eqref{linX}, we further bound as follows
\[
\begin{split}
\| \frac{\tu^2(\bu_t^2 + \bw^2)}{r} \|_{L^1 \cap L^2} & \les \| \frac{\tu}{\log(2+r)} \|_{L^2} \| \frac{\tu}{\log(2+r)} \|_{L^2 \cap L^\infty}
\| \frac{\bu_t^2 + \bw^2}{r} \log^2(2+r) \|_{L^\infty} \\
&  \lesss \frac{\log^2 t}{t^{5}} M^2
\end{split}
\]
\subsection{Conclusion} The proof of Proposition \ref{p:mn} is a direct consequence of all
the estimates in the previous subsections. Indeed, the result in part a) follows from \eqref{nbubw} and \eqref{Kest}.
The results in part b) follow from \eqref{nlinlx}, \eqref{nn} and \eqref{Kest}.

We are also ready to prove our main result.

\begin{proof}[Proof of Theorem \ref{t:main}]
Based on the results in Proposition \ref{p:mn}, one can iterate the equation \eqref{geq}
in the following space 
\[
\| \tw \|_{Y}=t^\frac32 \| \tw(t)\|_{LX} +
 t^\frac52 \| \partial_t \tw(t)\|_{LX} +
t^\frac52 \| \tw(t)\|_{\dot H^1} 
\]
The size of $\tu$ is controlled by using the results of Proposition \ref{propeg}. 
\end{proof}

\end{document}